\documentclass[10pt]{amsart}

\makeatletter
\def\@citecolor{blue}
\def\@urlcolor{blue}
\def\@linkcolor{blue}
\makeatother

\usepackage{graphicx}
\usepackage[headings]{fullpage}
\usepackage{lmodern}
\usepackage[T1]{fontenc}
\usepackage{amsmath,amsthm,amssymb,amsfonts,latexsym,mathrsfs}
\usepackage[all]{xy}
\usepackage{paralist}
\usepackage{color}
\usepackage{ifthen}
\usepackage{wrapfig}
\usepackage[normalem]{ulem}

\makeatletter

\@addtoreset{equation}{section}
\def\theequation{\thesection.\@arabic \c@equation}
\def\@citecolor{blue}
\def\@urlcolor{blue}
\def\@linkcolor{blue}
\def\theenumi{\@roman\c@enumi}

\makeatother

\theoremstyle{plain}
\newtheorem{theorem}[equation]{Theorem}
\newtheorem{lemma}[equation]{Lemma}
\newtheorem{corollary}[equation]{Corollary}
\newtheorem{proposition}[equation]{Proposition}

\theoremstyle{definition}
\newtheorem{remark}[equation]{Remark}

\newtheorem{remarks}[equation]{Remarks}

\newtheorem{example}[equation]{Example}

\newtheorem{definition}[equation]{Definition}

\def\NZQ{\mathbb}               
\def\NN{{\NZQ N}}

\def\frk{\mathfrak}               

\def\mm{{\frk m}}
\def\qq{{\frk q}}
\def\pp{{\frk p}}

\def\opn#1#2{\def#1{\operatorname{#2}}} 
\opn\chara{char}
\opn\length{\ell}
\opn\projdim{proj\,dim}
\opn\depth{depth}
\opn\reg{reg}
\opn\lreg{lreg}
\opn\sat{^{sat}}
\opn\lex{^{lex}}
\opn\Ker{Ker}
\opn\Coker{Coker}
\opn\Im{Im}
\opn\Hom{Hom}
\opn\Tor{Tor}
\opn\Ext{Ext}
\opn\End{End}
\opn\Aut{Aut}
\opn\id{id}
\opn\GL{GL}
\let\:=\colon

\let\lra\longrightarrow

\let\Ra\Rightarrow
\opn\Gin{Gin}
\opn\BW{BW}

\opn\Hilb{Hilb}
\opn\ini{in}
\opn\End{end}

\begin{document}
\title{A rigidity property of local cohomology modules}

\author{Enrico Sbarra}
\address{Enrico Sbarra - Dipartimento di Matematica - Universit\`a degli Studi di Pisa - Largo Bruno Pontecorvo 5 - 56127 Pisa - Italy}
\email{sbarra@dm.unipi.it}
\author{Francesco Strazzanti}
\address{Francesco Strazzanti - Dipartimento di Matematica - Universit\`a degli Studi di Pisa - Largo Bruno Pontecorvo 5 - 56127 Pisa - Italy}
\thanks{The first author was partially supported by PRA project 2015-16 ``Geometria, Algebra e
Combinatoria di Spazi di Moduli e Configurazioni'', University of Pisa.
  The work of the second author was supported by the ``National Group for Algebraic and Geometric Structures, and their Applications'' (GNSAGA-INDAM)}
\email{strazzanti@mail.dm.unipi.it}

\begin{abstract}
The relationships between the invariants and the homological properties of $I$, $\Gin(I)$ and $I\lex$  have been studied extensively over the past decades. A result of A. Conca, J. Herzog and T. Hibi points out some rigid behaviours of their Betti numbers. In this work we establish a local cohomology counterpart of their theorem. To this end, we make use of properties of sequentially Cohen-Macaulay modules and we study a generalization of such concept by introducing what we call partially sequentially Cohen-Macaulay modules, which might  be of interest by themselves.
\end{abstract}

\date{\today}

\maketitle

\section*{Introduction}
Let $I$ be a homogeneous ideal of $R=K[X_1, \dots,X_n]$ over a field $K$. Several results are known about the Betti numbers of $I$, in connection with its generic initial ideals and its lex-ideal. By \cite{Bi}, \cite{Hu} and \cite{Pa}, the maximal graded Betti numbers of a given Hilbert function are achieved by the unique lex-ideal with such Hilbert function; in particular this means that the Betti numbers of $I$ are always less than or equal to those of its lex-ideal $I\lex$. It is also well-known that the Betti numbers of $\ini_{\prec}(I)$ are greater than or equal to those of $I$, for any monomial order $\prec$. In particular, if one considers the generic initial ideal of $I$ with respect to reverse lexicographic order $\Gin(I)$,  $I$ and $\Gin(I)$ have the same Betti numbers if and only if $I$ is componentwise linear; this is proved in characteristic zero by A. Aramova, J. Herzog and T. Hibi in \cite{ArHeHi} and generalized in \cite{CaSb} to any characteristic.

We recall that the Betti numbers $\beta_{ij}(R/I)$ of $R/I$ are defined as the dimensions of the vector spaces of $\Tor_i^R(R/I,K)_j$. Thus, it is natural to ask whether similar results hold for $\dim_K \Ext_R^i(R/I,R)_j$ or, equivalently via local duality, for $h^k(R/I)_j:=\dim_K H^k_{\mm}(R/I)_j$, i.e. for the Hilbert function of the local cohomology modules with support on the maximal graded ideal of $R$. A first step in this direction has been made in \cite{Sb1}, where the first author proves that $h^k(R/I)_j \leq h^k(R/\ini_{\prec}(I))_j \leq h^k(R/I\lex)_j$ for any monomial order $\prec$ and for all $k$ and $j$. Moreover, in \cite{HeSb}, it is proven that  $h^k(R/I)_j = h^k(R/\Gin(I))_j$ for all $j$ and $k$ if and only if $R/I$ is a sequentially Cohen-Macaulay module; this notion was introduced independently by Schenzel in \cite{Sc} and Stanley in \cite{St} and has been widely studied, especially in Combinatorial Commutative Algebra  because of its connection with non-pure shellability. In this paper we introduce the notion of {\it partially sequentially Cohen-Macaulay} modules in order to characterize the ideals for which $h^k(R/I)_j = h^k(R/\Gin(I))_j$ for all $k \geq i$ and all $j$, see Section $3$ for more details.

Moreover, in \cite{HeHi}, J. Herzog and T. Hibi prove that $I$ and its lex-ideal have the same Betti numbers if and only if $\beta_{0j}(I)=\beta_{0j}(I\lex)$ for all $j$ - these ideals are the so-called {\it Gotzmann ideals} - whereas, in \cite{Sb}, it is shown  that the local cohomology modules of $I$ and $I\lex$ have the same Hilbert functions if and only if $h^0(R/I)_j=h^0(R/I\lex)_j$ for all $j$, and these ideals have been characterized as those such that $(I\sat)\lex=(I\lex)\sat$. We prove, in Theorem \ref{rigidit}, that this property is equivalent to $\Gin(I)\sat=(I\lex)\sat$ and in Proposition \ref{BW} that such ideals are exactly those for which $I$ and $I\lex$ have the same Bj\"orner-Wachs polynomial, a tool recently introduced by A. Goodarzi in \cite{Go} in order to characterize sequentially Cohen-Macaulay modules. 

Finally, in \cite{CoHeHi}, the aforementioned result of Herzog and Hibi is generalized to a rigidity property of Betti numbers as follows: if $\beta_{ij}(R/I)=\beta_{ij}(R/I\lex)$ for some $i$ and all $j$, then $\beta_{kj}(R/I)=\beta_{kj}(R/I\lex)$ for all $k \geq i$ and all $j$. Our main result is a similar statement about the Hilbert function of local cohomology modules; in fact, we prove more: if $h^i(R/\Gin(I))_j=h^i(R/I\lex)_j$ for some $i$ and all $j$, then $h^k(R/I)_j=h^k(R/I\lex)_j$ for all $k \geq i$ and all $j$. In particular, if $h^i(R/\Gin(I))_j=h^i(R/I\lex)_j$, one has $h^i(R/I)_j=h^i(R/I\lex)_j$; we notice that it is not true for Betti numbers, as showed in \cite{MuHiExample}. 

This paper is structured as follows. In the first section we recall some preliminary results; we also prove Proposition \ref{cancellations} about consecutive cancellations for local cohomology, which is easy, as well as several results about the ideals with maximal local cohomology, which are useful later in the article. In the second section we prove that the local cohomology modules of $R/I$ have the same Hilbert functions as those of $R/I\lex$  if and only if $I$ and $I\lex$ have the same Bj\"orner-Wachs polynomial, see Proposition \ref{BW}. In Section $3$ we introduce the notion of partially sequentially Cohen-Macaulay, see Definition \ref{ipsCM} and, in Theorem \ref{partiallySCM}, we prove a characterization that is crucial in the sequent and last section, where we prove our main result, see Theorem \ref{yeah^3}, and some of its consequences in Corollary \ref{lbutnol}.

\section{Preliminaries}

Let $R=K[X_1, \dots, X_n]$ be a standard graded polynomial ring in $n$ variables over a field $K$, which we may assume infinite without loss of generality, and $I$ be a homogeneous ideal of $R$. Given a graded $R$-module $M$, we write $M_d$ to indicate its $d$th homogeneous component; we also denote by $\Hilb(M)$ its Hilbert series,  by $\Hilb(M)_j$ the $j$th coefficient of its Hilbert series, i.e. the $j$th value of its Hilbert function, and by $P_{M}(t)$ its Hilbert polynomial.
Let  $H^i_{\mm}(M)$ denote the $i$th local cohomology module of $M$ with support on the maximal graded ideal $\mm=(X_1, \ldots,X_n)$; set $h^i(M):=\Hilb(H^i_{\mm}(M))$ and $h^i(M)_j:=\Hilb(H^i_{\mm}(M))_j=\dim_K H^i_{\mm}(M)_j$. We also let $R_{[j]}=K[X_1, \dots, X_j]$ and $I_{[j]}=I \cap R_{[j]}$. We denote the generic initial ideal of $I$ with respect to the reverse lexicographic order by $\Gin(I)$. A monomial ideal $I$ is said to be a {\it lex-ideal} if for any monomial $u\in I_d$ and all monomials $v\in R_d$ with $u\prec_{\lex} v$ one has $v\in I$. Given an ideal $I$, there exists a unique lex-ideal with the same Hilbert function as $I$ (cf. for instance  \cite[Thm. 6.3.1]{HeHi1}); we denote it by $I\lex$. The saturation of $I$ is the homogeneous ideal $I\sat := I:\mm^\infty = \cup_{k=1}^\infty (I:\mm^k)$. It is well-known that $\Gin(I)\sat=\Gin(I\sat)$. Given a monomial $u \in R$, we denote by $m(u)$ the maximum integer for which $X_{m(u)}$ divides $u$.  A monomial ideal $I$ is said to be {\it weakly stable} if for any monomial $u \in I$ and for all $j < m(u)$, there exists a positive integer $k$ such that $X_j^k u/X^l_{m(u)} \in I$, where $l$ is the largest integer such that $X^l_{m(u)}$ divides $u$. Notice that lex-ideals and generic initial ideals are weakly stable and recall that the saturation of weakly stable ideals can be computed by saturating with the last variable, i.e. $I\sat = I:X_n^\infty = \cup_{k=1}^\infty (I:X_n^k)$. Also observe that, if $I$ is weakly stable then so is $I_{[j]}$ for all $j$. Such ideals are also referred to as {\it Borel-type} ideals, see \cite[Section 4.2]{HeHi1}  for further reference.
Let now $M$ be any $R$-module. Let $\mathcal{M}_{-1}=0$ and, given a non-negative integer $k$, we denote by $\mathcal{M}_k$ the maximum submodule of $M$ with dimension less than or equal to $k$; we call $\{\mathcal{M}_k\}_{k\geq -1}$ the dimension filtration of $M$. The module $M$ is said to be {\it sequentially Cohen-Macaulay}, sCM for short, if $\mathcal{M}_k/\mathcal{M}_{k-1}$ is either zero or a $k$-dimensional Cohen-Macaulay module for all $k\geq 0$; in this case, if $M=R/I$, we simply say that $I$ is a sCM ideal. It is not difficult to see that, if $I$ is weakly stable, then it is a sCM ideal. 
 
Finally, it is well-known by \cite[Thm. 3.7]{Bi}, \cite[Thm. 2]{Hu}, \cite[Thm. 31]{Pa} and \cite[Thm. 2.4 and 5.4]{Sb1} that, for all $i$ and $j$,
\begin{gather}
\beta_{ij}(R/I) \leq \beta_{ij}(R/\Gin(I)) \leq \beta_{ij}(R/I\lex)\label{pip}\\ 
h^i(R/I)_j \leq h^i(R/\Gin(I))_j \leq h^i(R/I\lex)_j.\label{pop}
\end{gather}

\subsection{Universal lex-ideals and critical Hilbert functions}
A special class of lex-ideals which are of interest in the following is what we call, following \cite{MuHi1} and \cite{MuHi} which are our main references here, {\it universal lex-ideals}; these were first introduced in the squarefree case in \cite{BaNoTh}.
A universal lex-ideal is simply a lex-ideal with at most $n$ minimal generators. They are universal in the sense that they are exactly those lex-ideals whose extensions to any polynomial overring of $R$ are still lex-ideals. A numerical function $H\: \NN\lra \NN$ is said to be {\it critical} if it is the Hilbert function of an universal lex-ideal and, accordingly, a homogeneous ideal is called {\it critical} if its Hilbert function is. By \cite[Thm. 1.6]{MuHi}, we know that,  if $I$ is a critical ideal, then $\depth R/I=\depth R/I\lex =n-|G(I\lex)|$, where $G(I\lex)$ denotes the set of minimal generators of $I\lex$. Moreover, if a lex-ideal has positive depth, \cite[Cor. 1.4]{MuHi} yields that it is universal. 



\subsection{Consecutive cancellations in Hilbert functions of local cohomology modules}
Using the proof of \cite[Prop. 30]{Pa}, in \cite[Thm. 1.1]{Pe} it is proven that the graded Betti numbers of a homogeneous ideal can be obtained from the graded Betti numbers of its associated lex-ideal by a sequence of consecutive cancellations. Following the same line of reasoning, one can prove, and we do, an analogue for the Hilbert function of local cohomology modules.

Let $\{c_{i,j}\}$ be a set of natural numbers, where $(i,j) \in \NN^2$. Fix an index $j$ and choose $i$ and $i'$ such that one is even and the other is odd; then we obtain a new set by a {\it cancellation} if we replace $c_{i,j}$ by $c_{i,j}-1$ and $c_{i',j}$ by $c_{i',j}-1$. Such a cancellation is said to be  {\it consecutive} if $i'=i+1$. 
If $I$ and $J$ are two homogeneous ideals of $R$ with the same Hilbert function, from Serre formula, cf. \cite[Thm. 4.4.3 (b)]{BrHe}, we have that 
\begin{equation*} \label{Serre}
\sum_{i=0}^d (-1)^i \ h^i(R/I)_j = \Hilb(R/I)_j - P_{R/I}(j) =  \Hilb(R/J)_j - P_{R/J}(j) = \sum_{i=0}^d (-1)^i \ h^i(R/J)_j,
\end{equation*}
where $d=\dim R/I$.
By \cite[Thm. 2.4 and 5.4]{Sb1}, the above equalities imply that we can obtain the set $\{h^i(R/I)_j\}$ from both $\{h^i(R/\ini_\prec(I))_j\}$ and $\{h^i(R/I\lex)_j\}$ by a sequence of cancellations. In fact, the next result shows that the use of consecutive cancellations is enough.

\begin{proposition} \label{cancellations}
Let $\prec$ be a monomial order. The set $\{h^i(R/I)_j\}$ can be obtained from both the sets $\{h^i(R/\ini_\prec(I))_j\}$ and $\{h^i(R/I\lex)_j\}$ by a sequence of consecutive cancellations.  
\end{proposition}

\begin{proof}

By \cite[Prop. 30]{Pa} and its proof, there is a finite sequence of homogeneous ideals which starts with the ideal $I$ and terminates with the ideal $I\lex$; this is obtained by applying three types of basic operations: polarization, specialization by generic linear forms and taking initial ideals with respect to the lexicographic order. By the proof of \cite[Thm. 5.4]{Sb1}, we only need to check that $\{h^i(R/I)_j\}$ can be obtained by $\{h^i(R/\ini_{\prec}(I))_j\}$ by a sequence of consecutive cancellations. 
By the proof of \cite[Lemma 2.2]{Sb1} and the graded version of the Local Duality Theorem, it follows that there exist some non-negative integers $h_{i,j}$ such that for all $0 \leq i \leq d$ and all $j$ one has
\begin{equation*}
\begin{split}
h^i(R/\ini_{\prec}(I))_j & = h^i(R/I)_j + h_{n-i,j}+h_{n-i-1,j}, \\
\end{split}
\end{equation*}
where $h_{n-d-1,j}=0$. Therefore, since $0=\sum_{i=0}^d (-1)^i \ h^i(R/\ini_{\prec}(I))_j - \sum_{i=0}^d (-1)^i  h^i(R/I)_j=h_{n,j}$, the conclusion is now straightforward.
\end{proof}

\subsection{Maximality and rigidity results}
We conclude this preliminary section by recalling two rigidity results. Let us consider the two inequalities in \eqref{pip}. Recall that, in characteristic $0$, the first inequality is an equality if and only if $I$ is componentwise linear, as proved in \cite[Thm. 1.1]{ArHeHi}; see \cite[Thm. 2.9]{CaSb} for a generalization to any characteristic. It is also known that the three sets of Betti numbers coincide if and only if $I$ is a Gotzmann ideal, \cite[Cor. 1.4]{HeHi}. Moreover,  A. Conca, J. Herzog and T. Hibi proved the following theorem, which shows a rigidity property of the queue of a minimal free resolution.

\begin{theorem}[\cite{CoHeHi}, Cor. 2.7]\label{betarig}
Assume $\chara K=0$. Let $J$ be either the lex-ideal of $I$ or a generic initial ideal of $I$ and suppose that $\beta_{i} (R/I)= \beta_{i} (R/J)$ for some $i$. Then $\beta_k (R/I)= \beta_k (R/J)$ for all $k \geq i$. 
\end{theorem}

\noindent
Observe that the statement is equivalent to saying that, if $\beta_{ij} (R/I)= \beta_{ij} (R/J)$ for some $i$ and all $j$, then $\beta_{kj} (R/I)= \beta_{kj} (R/J)$ for all $k \geq i$ and all $j$.

Consider now the two inequalities in \eqref{pop}. By \cite[Thm. 3.1]{HeSb}, the first inequality is an equality if and only if $I$ is a sCM ideal. Moreover, in the next theorem we collect three conditions proved in \cite[Thm. 0.1]{Sb} that characterize the maximality of the Hilbert functions of local cohomology modules and a fourth one that is of crucial importance in this paper.

\begin{theorem}\label{rigidit}
 For any homogeneous ideal $I$, TFAE:

\medskip

{\bf 1}. $(I\sat)\lex=(I\lex)\sat$;

\smallskip

{\bf 2}. $h^0(R/I)_j=h^0(R/I\lex)_j$, for all $j$;

\smallskip

{\bf 3}. $h^i(R/I)_j=h^i(R/I\lex)_j$, for all $i,j$;

\smallskip

{\bf 4}. $\Gin(I)\sat=(I\lex)\sat$.
\end{theorem}

\begin{proof} We only need to show the equivalence between 4. and the other three conditions. \\
1. $\Rightarrow$ 4. Since $(I\sat)\lex$ is a saturated lex-ideal, it has positive depth and, therefore, it is an universal lex-ideal, i.e.  $I\sat$ is critical; therefore   
$\Gin(I)\sat=\Gin(I\sat)=(I\sat)\lex=(I\lex)\sat$, where the second equality holds by \cite[Lemma 2.6]{MuHi1}.\\ 
4. $\Rightarrow$ 1. Observe that the saturation of a lex-ideal is still a lex-ideal; therefore, $(I\lex)\sat=((I\lex)\sat)\lex=(\Gin(I)\sat)\lex$, which is $(\Gin(I\sat))\lex$. It is now enough to recall that  $I\sat$ and $\Gin(I\sat)$ have the same Hilbert function to obtain the desired conclusion.
\end{proof}

\begin{remark}\label{note} 
{\bf (a)} The equivalence of Conditions 1.-3. and Theorem 3.1 in \cite{HeSb} were proved when $\chara K=0$. Anyway, this hypothesis was used only because $\Gin(I)$ is not strongly stable in positive characteristic; since $\Gin(I)$ is  weakly stable in any characteristic though, it is easy to see that the original proofs work in general. Hence, throughout the paper, we do not need any assumption on the characteristic.\\
{\bf (b)} Under the hypothesis of Theorem \ref{betarig}, it is easy to find ideals which do not have maximal Betti numbers for all $k \leq i$ (see for instance \cite{MuHiExample}). Also, if $h^{i}(R/I)_j=h^{i}(R/I\lex)_j$ for all $j$, it is not true in general that $h^{k}(R/I)_j=h^k(R/I\lex)_j$ for all $k\leq i$ and all $j$.\\
{\bf (c)} An ideal $I$ which satisfies the equivalent conditions of Theorem \ref{rigidit} is a sCM ideal, see \cite[Prop. 1.9]{Sb}.\\
{\bf (d)} Conditions 1.-4. hold if and only if $h^i(R/\Gin(I))_j=h^i(R/I\lex)_j$ for all $i$ and $j$. Indeed, it is sufficient to recall that, by \cite[Lemma 2.3]{Co}, $\Gin(\Gin(I))=\Gin(I)$ and to apply 4.  \\
{\bf (e)} If $I$ is critical, then $\Gin(I)=I\lex$ and hence $h^i(R/I)_j=h^i(R/I\lex)_j$, for all $i,j$. Thus, $I$ is a sCM ideal. \\
{\bf (f)} It is proven in \cite[Thm. 1.6]{MuHi} that a critical ideal has the same depth of its lex-ideal and this is generalized by {\bf (e)}, since the $\depth R/I$ is the least integer $i$ such that $h^i(R/I)_j\neq 0$ for some $j$.
\end{remark}

\section{The Bj\"{o}rner-Wachs polynomial}

The Bj\"{o}rner-Wachs polynomial was introduced by A. Goodarzi in \cite{Go} in order to characterize  sCM ideals. These are exactly the ideals whose Bj\"{o}rner-Wachs polynomial does not change after taking the  generic initial ideal with respect to the reverse lexicographic order. 
In this section we prove a similar result, which characterizes those ideals whose Hilbert functions of local cohomoloy modules are maximal. First, we introduce some notations and recall a few results, see \cite{Go} for more details.

Let $I= \bigcap_{1}^s \qq_l$ be a reduced primary decomposition of $I$ and set $\pp_l=\sqrt{\qq_l}$ for all $1 \leq l \leq s$. Denote by $I^{\langle i \rangle}$ the ideal
$$
I^{\langle i \rangle} = \bigcap_{\dim R/\pp_l > i} \qq_l;$$
thus, $I^{\langle -1 \rangle}=I$ and $I^{\langle 0 \rangle}=\bigcap_{\pp_i \neq \mm} \qq_i = I\sat$. For $i=0,\ldots, d=\dim R/I$, we also denote  by $U_i(R/I)$  the $R$-module $I^{\langle i \rangle}/I^{\langle i-1 \rangle}$. We refer to such modules as the {\em  unmixed layers} of $R/I$; they are either $0$ or of dimension $i$.
We notice that,  $0 \subseteq I^{\langle 0 \rangle}/I \subseteq I^{\langle 1 \rangle}/I \subseteq \dots \subseteq I^{\langle d-1 \rangle}/I \subseteq R/I$ is the dimension filtration of $R/I$ and the modules $U_i(R/I)$ are the quotients that appear in the definition of sCM-module.

Let $h(M;t)$ be the $h$-polynomial of an $R$-module $M$ of dimension $d$, i.e. the numerator of its Hilbert series $\frac{h(M;t)}{(1-t)^d}$. The Bj\"{o}rner-Wachs polynomial of $R/I$, briefly BW-polynomial, is defined to be 
$$
\BW(R/I;t;w):=\sum_{k=0}^{\dim R/I} h(U_k(R/I);t)w^k.
$$
One of the  main results of \cite{Go} is that $\BW(R/I;t;w)=\BW(R/\Gin(I);t;w)$ if and only if $R/I$ is sCM.

\begin{proposition} \label{BW}
Let $I \subseteq R$ be a homogeneous ideal. The equivalent conditions of Theorem \ref{rigidit} hold if and only if 
$$
\BW(R/I;t;w)=\BW(R/I\lex;t;w).
$$
\end{proposition}

\begin{proof}
Suppose that $I$ and $I\lex$ have the same BW-polynomial, then, for all $i \geq 0$, $U_i(R/I)$ and $U_i(R/I\lex)$ have the same Hilbert function and the same holds for $I^{\langle i \rangle}$ and $(I\lex)^{\langle i \rangle}$. Therefore, for any $j$
$$h^0(R/I)_j=\Hilb(I^{\langle 0 \rangle})_j - \Hilb(I)_j=\Hilb((I\lex)^{\langle 0 \rangle})_j - \Hilb(I\lex)_j=h^0(R/I\lex)_j.$$

\noindent
Viceversa, let $I\lex=\cap_{j=1}^s\qq'_j$ be a reduced primary decomposition of $I\lex$; by assumption 
$$\Gin(I)^{\langle 0 \rangle}=\Gin(I)\sat=(I\lex)\sat=(I\lex)^{\langle 0 \rangle}=\bigcap_{\sqrt{q'_i} \neq \mm} \qq'_i.$$
Consequently, $\Gin(I)^{\langle i \rangle}=(I\lex)^{\langle i \rangle}$ for all $i\geq 0$, $U_i(\Gin(I))=U_i(I\lex)$ for $i=1, \dots, d-1$ and, therefore, 
$\Hilb(U_i(\Gin(I)))=\Hilb(U_i(I\lex))$ for $i\geq 0$. Thus,  $\Gin(I)$ and $I\lex$ have the same BW-polynomial. Moreover, by Remark \ref{note} (c), $I$ is a sCM ideal and, consequently, $I$ and $\Gin(I)$ have the same BW-polynomial by \cite[Thm. 17]{Go}.
\end{proof}

As a by-product of the previous proof, we get the following corollary.

\begin{corollary}
Let $I \subseteq R$ be a homogeneous ideal and suppose that $I$ and $I\lex$ have the same Bj\"{o}rner-Wachs polynomial $\BW$. Then,

\medskip

{\bf (i)}  $\Gin(I)^{\langle i \rangle}=(I\lex)^{\langle i \rangle}$ for all $i=0, \dots, d-1$.

\smallskip

{\bf (ii)} $\BW(R/\Gin(I);t;w)= \BW$, i.e. $I$ is a sCM ideal. 

\end{corollary}

We also notice that the only if part of Proposition \ref{BW} is yielded by \cite[Thm. 20]{Go}.

\section{Partially sequentially Cohen-Macaulay modules}

As we already mentioned before, for all $i, j$ one has $h^i(R/I)_j \leq h^i(R/\Gin(I))_j$ and all such inequalities are equalities if and only if $I$ is a sCM ideal. Thus, one may ask whether also a result like Theorem \ref{betarig} holds: if $h^{i}(R/I)_j = h^{i}(R/\Gin(I))_j$ for all $j$, is it true that $h^{k}(R/I)_j = h^{k}(R/\Gin(I))_j$ for all $k\geq i$ and all $j$? It is easy to see that this is not the case, even if $i=0$. Indeed, if one considers a non-sCM ideal with positive depth $t$, then $h^i(R/I)_j=h^i(R/\Gin(I))_j=0$ for all $0\leq i < t$ and all $j$, but there exists at least one index $i$ for which $h^i(R/I)_j \neq h^i(R/\Gin(I))_j$ for some $j$, since $I$ is not a sCM ideal. Moreover, $\Gin(I\sat)=\Gin(I)\sat$ yields immediately that, for any ideal $I$, $h^0(R/I)_j=h^0(R/\Gin(I))_j$ for all $j$.

In this section we introduce the notion of {\em partially sequentially Cohen-Macaulay} modules which, as we shall see in Theorem \ref{partiallySCM}, naturally  characterize those ideals for which $h^k(R/I)_j =  h^k(R/\Gin(I))_j$ for all $k$ larger than some homological index $i$. We do it in the next definition. 

\begin{definition}\label{ipsCM}
Let $i$ be a non-negative integer. A finitely generated $R$-module $M$ with dimension filtration $\{\mathcal{M}_k\}_{k\geq -1}$ is called $i$-partially sequentially Cohen-Macaulay, $i$-sCM for short, if $\mathcal{M}_k/\mathcal{M}_{k-1}$ is either zero or a $k$-dimensional Cohen-Macaulay module for all $i \leq k \leq \dim M$. 
\end{definition}

Clearly, with this notation, sequentially Cohen-Macaulay modules are exactly the $0$-sCM modules.  Several known results about sCM modules can be easily generalized to our context. In the following two lemmata we collect some properties of this kind; the proofs follow the same line of the original ones, which can be found in \cite{Sc} and \cite{Go} respectively. As in the previous section, for a homogeneous ideal $J\subseteq R$, we denote its unmixed layers  by $U_{\bullet}(R/J)$.

\begin{lemma} \label{regular element}
Let $M$ be a finitely generated $R$-module with dimension filtration $\{\mathcal{M}_k\}_{k\geq -1}$. 

\medskip

{\bf 1}. If $M$ is $i$-sCM, then $H^k_\mm (M) \cong H^k_\mm(\mathcal{M}_k) \cong H^k_\mm (\mathcal{M}_k/\mathcal{M}_{k-1})$ for all $k \geq i$; 

\smallskip

{\bf 2}. If $x \in R$ is an  $M$-regular element, then $M$ is $i$-sCM if and only if $M/xM$ is $(i-1)$-sCM; 

\smallskip

{\bf 3}. If $I$ is a homogeneous ideal, then $R/I$ is $i$-sCM if and only if $R/I\sat$ is $i$-sCM.
\end{lemma}

\begin{lemma} \label{Afshin}
Let $I$ be a homogeneous ideal. 

\medskip

{\bf 1}. $R/I$ is $i$-sCM if and only if $\Hilb(U_k(R/I))=\Hilb(U_k(R/\Gin(I)))$ for all $k \geq i$; 

\smallskip

{\bf 2}. if $R/I$ is $i$-sCM, then $h^k(R/I)_j=\Hilb(U_k(R/I))_j$ for all $k \geq i$ and all $j$.
\end{lemma}

\begin{remark} \label{remBW}
In the light of the previous lemma, if one lets the $i$th truncated Bj\"{o}rner-Wachs polynomial of $R/I$ be 
$$
\BW^i(R/I;t;w):=\sum_{k=i}^{\dim R/I} h(U_k(R/I);t)w^k,
$$

then  the following generalization of \cite[Thm. 17]{Go} holds:
$$R/I \text{ is } i\text{-sCM if and only if } \BW^i(R/I;t;w)=\BW^i(R/\Gin(I);t;w).$$
\end{remark}
 
Now, let $I$ be any ideal of $R$ and $l\in R_1$ be a generic linear form which, without loss of generality we may write as $l=a_1X_1 + \dots + a_{n-1}X_{n-1}-X_n$. Consider the map $g_n\: R \longrightarrow R_{[n-1]}$, defined by $X_i\mapsto X_i$ for $i=1,\ldots,n-1$ and $X_n\mapsto a_1X_1 + \dots + a_{n-1}X_{n-1}$. Then, the surjective homomorphism $\frac{R}{I} \longrightarrow \frac{R_{[n-1]}}{g_n(I)}$
has kernel $(I+(l))/I$  and induces the isomorphism  
\begin{equation}\label{generic}
\frac{R}{I+(l)} \cong \frac{R_{[n-1]}}{g_n(I)}.
\end{equation}

\noindent
Since $\Gin(I)$ is a monomial ideal, the image of $\Gin(I)$ in $R_{[n-1]}$ via the mapping $X_n \mapsto 0$ is $\Gin(I)_{[n-1]}$. With this notation, 
\cite[Cor. 2.15]{Gr} states that 
\begin{equation}\label{215}
\Gin(g_n(I))=\Gin(I)_{[n-1]}.
\end{equation}

The next theorem provides a useful characterization of partially sequentially Cohen-Macaulay modules that we are going to use in the next section.

\begin{theorem}\label{partiallySCM}
Let $I$ be a homogeneous ideal of $R$. TFAE:

\medskip

{\bf 1.} $R/I$ is $i$-sCM;

\smallskip

{\bf 2.}
$h^k(R/I)_j=h^k(R/\Gin(I))_j$ for all $k\geq i$ and for all $j$.
\end{theorem}

\begin{proof}
$1. \Rightarrow 2.$ is a direct consequence of Lemma \ref{Afshin}, because also $R/\Gin(I)$ is $i$-sCM.

\smallskip

$2. \Rightarrow 1.$ We prove the converse by induction on $d=\dim R/I$. If $d=0$ the ring is Cohen-Macaulay and then sCM. Therefore, without loss of generality we may assume that $R/I$ and $R/\Gin(I)$ have positive dimension and, by saturating $I$ if necessary, positive depth, cf. Lemma \ref{regular element}.3. Since $R/\Gin(I)$ is sCM,  \cite[Thm. 1.4]{HeSb} implies that there exists a linear form $l'$ which is $R/\Gin(I)$- and $\Ext^{n-k}_R(R/\Gin(I),\omega_R)$-regular for all $k$. Here $\omega_R$ denotes the canonical module of $R$. Recall that a change of coordinates does not affect the computation of the generic initial ideal; therefore, we may as well assume that $X_n$ is $R/\Gin(I)$-\, and $\Ext^{n-k}_R(R/\Gin(I),\omega_R)$-regular for all $k$. Thus, for all $k$, the short exact sequence $0 \rightarrow R/\Gin(I) (-1) \rightarrow R/\Gin(I) \rightarrow R/(\Gin(I)+(X_n)) \rightarrow 0$, gives raise  via Local Duality to short exact sequences
$$
0 \rightarrow H^{k-1}_{\mm}(R/(\Gin(I)+(X_n))) \rightarrow H^k_{\mm}(R/\Gin(I))(-1) \rightarrow H^k_{\mm}(R/\Gin(I)) \rightarrow 0.
$$ 
Thus, $h^{k-1}(R/(\Gin(I)+(X_n)))=(t-1)h^k(R/\Gin(I))$, for all $k$.

We also know that there exists a generic linear form $l$ which is $R/I$-regular; therefore, for all $k$,  the exact sequences 
$$
0 \rightarrow K \rightarrow H^{k-1}_{\mm}(R/(I+(l))) \rightarrow H^k_{\mm}(R/I)(-1) \rightarrow H^k_{\mm}(R/I) \rightarrow C \rightarrow 0
$$
imply that  $h^{k-1}(R/(I+(l))) \geq (t-1)h^k(R/I)$ for all $k$.

By \eqref{generic} and \eqref{215}, we thus have
\begin{equation*}\begin{split}
(t-1)h^k(R/I) &\leq  h^{k-1}(R/(I+(l)))=h^{k-1}(R_{[n-1]}/g_n(I)) \\
& \leq  h^{k-1}(R_{[n-1]}/\Gin(g_n(I)))
=h^{k-1}(R_{[n-1]}/\Gin(I)_{[n-1]})\\
&= h^{k-1}(R/(\Gin(I)+(X_n)))=(t-1)h^k(R/\Gin(I)),
\end{split}
\end{equation*}
where all of the above inequalities are equalities for all  $k \geq i$ by hypothesis.  In particular, $h^{k}(R_{[n-1]}/g_n(I)) = h^{k}(R_{[n-1]}/\Gin(g_n(I)))$ for all $k \geq i-1$, together with the inductive assumption, imply that $R_{[n-1]}/g_n(I) \simeq R/(I+(l))$ is $(i-1)$-sCM. Consequently, $R/I$ is $i$-sCM by Lemma \ref{regular element}.2, because $l$ is $R/I$-regular.
\end{proof}

\section{The rigidity property}

In this section we shall prove our main result, which establishes the desired analogue to Theorem \ref{betarig} and generalizes Theorem \ref{rigidit}. We start with some preliminary results.

  
\begin{lemma}
Let $I$ and $I'$ be two weakly stable ideals. If they have the same Hilbert polynomial, then $I_{[n-i]}$ and $I'_{[n-i]}$ have the same Hilbert polynomial for all $i=0, \dots, n$. Moreover the ideals $(I_{[n-i+1]}:X_{n-i+1}^{\infty})_{[n-i]}$ and $(I'_{[n-i+1]}:X_{n-i+1}^{\infty})_{[n-i]}$ have the same Hilbert polynomial.
\end{lemma}

\begin{proof}
  Recall that the saturation of a weakly stable ideal can be computed by saturating with respect to the last variable, and that saturation does not change the Hilbert polynomial. Moreover, by \cite[Lemma 1.4]{CaSb}, $I_{[n-1]}:X_{n-1}^{\infty} = (I:X_n^{\infty})_{[n-1]}:X_{n-1}^{\infty}$. 
Now the proof of the lemma is a straightforward consequence of the above.
\end{proof}

\begin{proposition} \label{weakrigidity}
Let $I \subseteq R$ be a weakly stable ideal. If $h^i(R/I)_j=h^i(R/I\lex)_j$ for some $i \geq 0$ and all $j$, then $h^k(R/I)_j=h^k(R/I\lex)_j$ for all $k \geq i$ and all $j$.  
\end{proposition}

\begin{proof}
If $i=0$ we know that the statement is true by Theorem \ref{rigidit} and, thus, we may assume $i \geq 1$. By Proposition \ref{cancellations}, the set $\{h^k(R/I)_j\}$ can be obtained by $\{h^k(R/I\lex)_j\}$ by means of a sequence of consecutive cancellations. Since at level $i$ there is nothing to be cancelled, the set $\{h^k(R/I)_j\}_{k \geq i}$ can be obtained from $\{h^k(R/I\lex)_j\}_{k \geq i}$ by a sequence of consecutive cancellations. In particular this implies that
$$
\sum_{k=i}^{n} (-1)^k \ h^k(R/I)_j = \sum_{k=i}^{n} (-1)^k \ h^k(R/I\lex)_j.
$$
Set $J=(I_{[n-i+1]}:X_{n-i+1}^{\infty})_{[n-i]}$ and $J'=((I\lex)_{[n-i+1]}:X_{n-i+1}^{\infty})_{[n-i]}$. Serre formula and \cite[Lemma 1.5]{CaSb} now imply 

$$
\Hilb(R_{[n-i]}/J)_j-P_{R_{[n-i]}/J}(j)= \sum_{k=0}^{n-i} (-1)^k \ h^k(R_{[n-i]}/J)_j = \phantom{aa}
$$
$$
\phantom{aa} 
= \sum_{k=0}^{n-i} (-1)^k \ h^k(R_{[n-i]}/J')_j= \Hilb(R_{[n-i]}/J')_j-P_{R_{[n-i]}/J'}(j);
$$
By the previous lemma, $J$ and $J'$ have the same Hilbert polynomial and, therefore, the same Hilbert function; since $J'$ is a lex-ideal, we conclude that $J'=J\lex$. Moreover, the ideal $(I\lex)_{[n-i+1]}:X_{n-i+1}^{\infty}$ is a  saturated lex-ideal of $R_{[n-i+1]}$ and, therefore, $J'$ has at most $n-i$ minimal generators. Hence, $J'$ is an universal lex-ideal and $J$ is a critical ideal. By Remark \ref{note} (e), we get  $h^k(R_{[n-i]}/J)_j=h^k(R_{[n-i]}/J')_j$ for $k=0, \dots, n-i$ and, by  \cite[Lemma 1.5]{CaSb}, this is enough to imply the conclusion.
\end{proof}




It might be useful to re-state what we have just proved as follows.

\begin{corollary} \label{caracterisescion}
Let $I$ be a weakly stable ideal and let $i$ be a positive integer. \\
Set $J=(I_{[n-i+1]}:X_{n-i+1}^{\infty})_{[n-i]}$ and $J'=((I\lex)_{[n-i+1]}:X_{n-i+1}^{\infty})_{[n-i]}$. TFAE: 

\medskip

{\bf (i)} $h^i(R/I)_j=h^i(R/I\lex)_j$ for all $j$; 

\smallskip

{\bf (ii)} $h^k(R/I)_j=h^k(R/I\lex)_j$ for all $k \geq i$ and all $j$;

\smallskip

{\bf (iii)} $J$ and $J'$ have the same Hilbert function; 

\smallskip

{\bf (iv)} $J$ is a critical ideal; 

\smallskip

{\bf (v)} $\Gin(J)=J'$.
\end{corollary}
\begin{proof}
  In the proof of the proposition we showed that (i) $\Ra$ (iii) $\Ra$ (iv) $\Ra$ (ii), which in turn obviously implies (i). Condition (iii) descends immediately by (v), whereas (iv) $\Ra$ (v) follows by \cite[Lemma 2.6]{MuHi1}, since $J'=J\lex$.
\end{proof}

We are now ready to prove the main theorem of the paper.

\begin{theorem}\label{yeah^3}
Let $I$ be a homogeneous ideal and let $i$ be a non-negative integer such that 
$h^i(R/\Gin(I))_j=h^i(R/I\lex)_j$ for all $j$. Then $h^k(R/I)_j=h^k(R/I\lex)_j$ for all $k \geq i$ and all $j$. 
\end{theorem}


\begin{proof}
The case $i=0$ is yielded by Remark \ref{note} (d); we thus assume that $i>0$.
Since $\Gin(I)$ is a weakly stable ideal, by Proposition \ref{weakrigidity} we have that $h^k(R/\Gin(I))_j=h^k(R/I\lex)_j$ for all $k \geq i$ and all $j$; hence, by Theorem \ref{partiallySCM}, it is enough to prove that $R/I$ is $i$-sCM.

Since $R/I\sat$ has positive depth, we know that there exists a generic linear form $l_n$ which  is $R/I\sat$-regular. Now, by \eqref{generic}, $R_{[n-1]}/g_n(I\sat)\simeq R/(I\sat + (l_n))$; thus, if we let $J_1:=g_n(I\sat)$, by \eqref{215} we have that $\Gin(J_1)=(\Gin(I)\sat)_{[n-1]}$ and $R/I$ is $i$-sCM if and only if $R_{[n-1]}/J_1$ is $(i-1)$-sCM. 
If $i-1 > 0$ we may continue in this way: we saturate $J_1$ with respect to the maximal ideal of $R_{[n-1]}$, we take a generic linear form $l_{n-1}\in R_{[n-1]}$ which is $R_{[n-1]}/J_1\sat$-regular and apply \eqref{generic}.  By letting $J_2$ be the ideal $g_{n-1}(J_1\sat)$, we have  $R_{[n-2]}/J_2\simeq R_{[n-1]}/(J_1\sat + (l_{n-1}))$ and $$\Gin(J_2)=(\Gin(J_1)\sat)_{[n-2]}=(((\Gin(I)\sat)_{[n-1]})\sat)_{[n-2]}=((\Gin(I)_{[n-1]})\sat)_{[n-2]},$$ where the last equality holds by \cite[Lemma 1.4]{CaSb}. After $i$ steps, we shall have 
\begin{equation}\label{gingin}
\Gin(J_i)=((\Gin(I)_{[n-i+1]})\sat)_{[n-i]}
\end{equation}
and $R/I$ is $i$-sCM if and only if  $R_{[n-i]}/J_i$ is $0$-sCM. Since $\Gin(I)$ is a weakly stable ideal and $h^i(R/\Gin(I))=h^i(R/I\lex)$, by Corollary \ref{caracterisescion} we have that $((\Gin(I)_{[n-i+1]})\sat)_{[n-i]}$ is a critical ideal. By \eqref{gingin} also $J_i$ is a critical ideal. Remark \ref{note} (e) thus implies that $R_{[n-i]}/J_i$ is sCM, as desired.
\end{proof}


\begin{corollary}\label{lbutnol}
Let $I$ be a homogeneous ideal and $i$ be a positive integer.

\noindent
 Set $J=$ $(\Gin(I)_{[n-i+1]}:X_{n-i+1}^{\infty})_{[n-i]}$ and $J'=((I\lex)_{[n-i+1]}:X_{n-i+1}^{\infty})_{[n-i]}$. TFAE: 

\medskip

{\bf (i)} $h^i(R/\Gin(I))_j=h^i(R/I\lex)_j$ for all $j$; 

\smallskip

{\bf (ii)} $h^i(R/I)_j=h^i(R/I\lex)_j$ for all $j$; 

\smallskip

{\bf (iii)} $h^k(R/I)_j=h^k(R/I\lex)_j$ for all $k\geq i$ for all $j$;

\smallskip

{\bf (iv)} $J$ and $J'$ have the same Hilbert function; 

\smallskip

{\bf (v)} $J$ is a critical ideal; 

\smallskip

{\bf (vi)} $\Gin(J)=J'$.
\end{corollary}

\begin{remark}
{\bf (a)} Theorem \ref{yeah^3} implies that if $h^i(R/\Gin(I))_j=h^i(R/I\lex)_j$ for all $j$, then $h^i(R/I)_j=h^i(R/I\lex)_j$ for all $j$; this is not true for the Betti numbers, see \cite[Thm. 3.1]{MuHiExample}. \\ 
{\bf (b)} As in Remark \ref{remBW}, it is straightforward that $I$ and $I\lex$ have the same $i$th truncated BW-polynomial if and only if $h^i(R/\Gin(I))_j=h^i(R/I\lex)_j$ for all $j$. \\
{\bf (c)} In \cite{CaSb} the notion of zero-generic initial ideal has been introduced. The zero-generic initial ideal $\Gin_0(I)$  shares with the usual one $\Gin(I)$ many of its most interesting properties and the two notions coincide in characteristic $0$. We observe that all the equivalent conditions of the previous result are still valid for $\Gin_0(I)$. In fact, since $h^i(R/\Gin_0(I))=h^i(R/\Gin(I))$ for all $i$, Theorem \ref{yeah^3} clearly holds for $\Gin_0(I)$, and, since $\Gin_0(I)$ is weakly stable, the conclusions of Corollary \ref{caracterisescion} as well. Thus, we are left to show that $\Gin_0(J)=J'$ is equivalent to conditions (i)-(v). One direction is immediately seen, since if $\Gin_0(J)=J'$, they have the same Hilbert function; the ideal $J'$ is a universal lex-ideal and, thus, $J$ is critical. Conversely, if $J$ is critical, so is $\Gin_0(J)$. Therefore, $\Gin(\Gin_0(J))=\Gin_0(J)\lex$. Since $\Gin_0(J)$ is Borel-fixed, the ideal on the left is $\Gin_0(J)$, whereas the ideal on the right is $J'=J\lex$, because $\Gin_0(J)$ and $J$ have the same Hilbert function.
\end{remark}

\noindent
{\bf Acknowledgments}. The authors would like to thank A. Conca and A. Goodarzi for some useful discussions about the topics of this paper.
%

\end{document}